\documentclass{elsart}
\usepackage{amsmath}
\usepackage{amsfonts}
\usepackage{amssymb}
\usepackage{graphicx}

\newenvironment{proof}[1][Proof: ]{\begin{trivlist}\item[\hskip \labelsep {\bfseries #1}]}{\QED \end{trivlist}}
\newcommand{\QED}{\hfill$\square$}

\begin{document}

    \begin{frontmatter}
        \title{On the ordering of trees by the Laplacian coefficients}
        \author[PMF]{Aleksandar Ili\' c},
        \ead{aleksandari@gmail.com}
        \address[PMF]{Faculty of Sciences and Mathematics, University of Ni\v s, Serbia}
        \journal{Linear Algebra and its Applications}

        \begin{abstract}
        We generalize the results from [X.-D. Zhang, X.-P. Lv, Y.-H. Chen,
        \textit{Ordering trees by the Laplacian coefficients}, Linear Algebra
        Appl. (2009), doi:10.1016/j.laa.2009.04.018] on the partial ordering of trees with
        given diameter. For two $n$-vertex trees $T_1$ and $T_2$, if $c_k (T_1) \leqslant c_k (T_2)$
        holds for all Laplacian coefficients $c_k$, $k = 0, 1, \ldots, n$, we say that $T_1$ is dominated
        by $T_2$ and write $T_1 \preceq_c T_2$. We proved that among
        $n$ vertex trees with fixed diameter $d$, the caterpillar $C_{n, d}$ has minimal Laplacian coefficients $c_k$,
        $k = 0, 1, \ldots, n$.
        The number of incomparable pairs of trees on $\leqslant 18$ vertices is presented, as well as infinite families of
        examples for two other partial orderings of trees, recently proposed by Mohar.
        For every integer $n$, we construct a chain $\{T_i\}_{i = 0}^m$ of $n$-vertex trees of length $\frac{n^2}{4}$,
        such that $T_0 \cong S_n$, $T_m \cong P_n$ and $T_i \preceq_c T_{i + 1}$ for all $i = 0, 1, \ldots, m - 1$.
        In addition, the characterization of the partial ordering of starlike trees is established by the majorization inequalities
        of the pendent path lengths.
        We determine the relations among the extremal trees with fixed maximum degree,
        and with perfect matching and further support the Laplacian coefficients
        as a measure of branching.

        \end{abstract}

    \end{frontmatter}

\section{Introduction}

Let $G = (V, E)$ be a simple undirected graph with $n = |V|$
vertices and $m = |E|$ edges. The Laplacian polynomial $P
(G,\lambda)$ of $G$ is the characteristic polynomial of its
Laplacian matrix $L (G) = D (G) - A(G)$,
$$
P (G, \lambda) = \det (\lambda I_n - L (G)) = \sum_{k = 0}^n (-1)^k
c_k \lambda^{n - k}.
$$

The Laplacian matrix $L(G)$ has non-negative eigenvalues $\mu_1
\geqslant \mu_2 \geqslant \ldots \geqslant \mu_{n - 1} \geqslant
\mu_n = 0$. From Viette's formulas, $c_k = \sigma_k (\mu_1, \mu_2,
\ldots, \mu_{n - 1})$ is a symmetric polynomial of order $n - 1$. In
particular, we have $c_0 = 1$, $c_n = 0$, $c_1 = 2m$, $c_{n - 1} =
n\tau (G)$, where $\tau (G)$ denotes the number of spanning trees of
$G$ (see \cite{CvDS95} and \cite{Me95}). If $G$ is a
tree, the coefficient $c_{n - 2}$ is equal to its Wiener index, which is
a sum of distances between all pairs of vertices
$$
c_{n - 2} (T) = W (T) = \sum_{u, v \in V} d (u, v),
$$
while the coefficient $c_{n - 3}$ is its modified hyper-Wiener index, introduced by
Gutman in \cite{Gu03}. The Wiener index is considered as one of the most used topological
indices with high correlation with many physical and chemical
properties of molecular compounds (for recent results and
applications of Wiener index see \cite{DoEn01}).

Let $m_k (G)$ be the number of matchings of $G$ containing exactly
$k$ independent edges. The subdivision graph $S (G)$ of $G$ is
obtained by inserting a new vertex of degree two on each edge of~$G$.
Zhou and Gutman \cite{ZhGu08} proved that for every acyclic graph~$T$
with $n$ vertices holds
\begin{equation}
\label{eq-zhou} c_k (T) = m_k (S (T)), \quad 0 \leqslant k \leqslant n.
\end{equation}

Let $T_1$ and $T_2$ be two trees of order $n$.
Denote by $r$ (respectively $s$) the smallest
(respectively largest) integer such that $c_r (T_1) \neq c_r (T_2)$ (respectively $c_s (T_1) \neq c_s(T_2)$).
Two partial orderings may be
defined as follows. If $c_r(T_1) < c_r(T_2)$, we say that $T_1$ is \emph{smaller than} $T_2$ and denote $T_1 \prec^1 T_2$.
If $c_s(T_1) < c_s(T_2)$, we say that $T_1$ is \emph{smaller than} $T_2$ and denote $T_1 \prec^2 T_2$.
We may now introduce relations $\prec_c$ and $\preceq_c$ on the set of $n$-vertex graphs by defining
$$
G \preceq_c H \qquad \Leftrightarrow \qquad c_k (G) \leqslant c_k (H), \ \ k = 0, 1, \ldots, n.
$$
and
$$
G \prec_c H \quad \Leftrightarrow \quad G \preceq_c H \ \ \mbox{and} \ \ c_k (G) < c_k (H) \ \ \mbox{for some} \ \ 1 \leqslant k \leqslant n - 1.
$$

Recently, Mohar on his homepage \cite{Mo06} proposed some problems on ordering trees
with the Laplacian coefficients.

\begin{prob}
Do there exist two trees $T_1$ and $T_2$ of order $n$ such that $T_1 \prec^1 T_2$ and $T_2 \prec^2 T_1$?
\end{prob}

\begin{prob}
Do there exist two trees $T_1$ and $T_2$ of order $n$ such that $T_1 \prec^1 T_2$ and $T_1 \prec^2 T_2$,
but there is an index $i$ such that $c_i(T_1) > c_i(T_2)$?
\end{prob}

\begin{prob}
Let $\mathfrak{T}_n$ be the set of all trees of order $n$.
How large chains and antichains of pairwise non-Laplacian-cospectral trees are there?
\end{prob}

\begin{prob}
Let us define $U (T, T')$ to be the set of all trees $Z$ of order $n = |T| = |T'|$ such that $Z$ majorizes $T$ and $T'$ simultaneously.
For which trees $T$ and $T'$ has $U (T,T')$ only one minimal element up to cospectrality, i.e., when are all minimal elements in $U (T,T')$ cospectral?
\end{prob}

Our goal here is also to add some further evidence to support
the use of Laplacian coefficients as a measure of branching in alkanes.
A topological index acceptable as a measure of branching must satisfy the inequalities \cite{FiGuHRVV02}
$$
TI (P_n) < TI (X_n) < TI (S_n) \qquad \mbox{or} \qquad TI (P_n) > TI (X_n) > TI (S_n),
$$
for $n = 4, 5, \ldots$, where $P_n$ is the path, and $S_n$ is the star on $n$ vertices.
For example, the first relation is obeyed by the largest graph eigenvalue and Estrada index,
while the second relation is obeyed by the Wiener index, Hosoya index and graph energy.
It is proven in \cite{Mo07} and \cite{ZhGu08} that for arbitrary tree $T \not \cong P_n, S_n$ holds
$$
c_k (P_n) > c_k (T) > c_k (S_n),
$$
for all $2 \leqslant k \leqslant n - 2$. We further refine this relation, by introducing
long chain of inequalities.

Stevanovi\' c and Ili\' c in \cite{StIl08} investigated the properties of the Laplacian
coefficients of unicyclic graphs. Guo in \cite{Gu03} presented the several tree orderings by the Laplacian spectral radius, while
Dong and Guo in \cite{DoGu06} used Wiener index for ordering the trees.
The authors in \cite{IlIlSt09} generalized the recent results from \cite{LiPa08} and \cite{WaGu08},
which proved that the caterpillar $C_{n, d}$ is the unique tree with $n$ vertices and diameter $d$,
that minimizes Wiener index $c_{n - 2}$.
X.-D. Zhang et al. in \cite{ZhLvCh09}, proved that $C_{n, d}$ has minimal Laplacian coefficients only for the
cases $d = 3$ and $d = 4$, while here we prove it for all $2 \leqslant d \leqslant n - 1$.

The paper is organized as follows. In Section 2 we revise two graph transformations, such that
all Laplacian coefficients are monotone under these transformations. Also, we derive
the partial ordering of starlike trees based on the majorization inequalities of the pendent path lengths.
In Section 3 we give an alternative proof of the fact that among
$n$ vertex trees with fixed diameter $d$, the caterpillar $C_{n, d}$ has minimal Laplacian coefficient $c_k$,
for every $k = 0, 1, \ldots, n$.
In addition, we construct an infinite family of incomparable pairs of trees based on two Mohar's ordering and
calculate the Laplacian coefficients for the special case $d = n - 3$.
In Section 4, the number of incomparable pairs of trees for $n \leqslant 18$ vertices is presented, and we also
derive a chain of inequalities of length $m \sim \frac{n^2}{4}$, such that
$$
S_n \cong T_0 \preceq_c T_1 \preceq_c T_2 \preceq_c \ldots \preceq_c T_{m - 1} \preceq_c T_m \cong P_n.
$$
In Section 5, we deal with the extremal Laplacian coefficients of trees with given maximum degree and perfect matching.

\section{Transformations and starlike trees}

Mohar in \cite{Mo07} proved that every tree can be transformed into a star by a sequence of $\sigma$-transformations.
Here we present the transformation from \cite{IlIl09}, that is a generalization of $\sigma$-transformation.

\begin{defn}
Let $v$ be a vertex of a tree $T$ of degree $m + 1$. Suppose that
$P_1, P_2, \ldots, P_m$ are pendent paths incident with $v$, with
lengths $n_i \geqslant 1$, $i = 1, 2, \ldots, m$. Let $w$ be the
neighbor of $v$ distinct from the starting vertices of paths $v_1,
v_2, \ldots, v_m$, respectively. We form a tree $T' = \delta (T, v)$
by removing the edges $v v_1, v v_2, \ldots, v v_{m - 1}$ from $T$
and adding $m - 1$ new edges $w v_1, w v_2, \ldots, w v_{m - 1}$
incident with $w$. We say that $T'$ is a $\delta$-transform of $T$.
\end{defn}

This transformation preserves the number of pendent vertices in a
tree $T$ and decreases all Laplacian coefficients.

\begin{thm}
\label{thm-delta} Let $T$ be an arbitrary tree, rooted at the center
vertex. Let vertex $v$ be a vertex furthest from the center of tree $T$ among all
branching vertices with degree at least three. Then, for
$\delta$-transformation tree $T' = \delta (T, v)$ and $0 \leqslant k
\leqslant n$ holds
$$
\label{eqn-delta} c_k (T) \geqslant c_k (T').
$$
\end{thm}

Mohar in \cite{Mo07} proved that every tree can be transformed into a path by a sequence of $\pi$-transformations.
Here we present the transformation from \cite{IlIlSt09}, that is a generalization of $\pi$-transformation.

\begin{thm}
\label{thm-pi} Let $w$ be a vertex of the nontrivial connected graph
$G$ and for nonnegative integers $p$ and $q$, let $G (p, q)$ denote
the graph obtained from $G$ by attaching pendent paths $P = w v_1
v_2 \ldots v_p$ and $Q = w u_1 u_2 \dots u_q$ of lengths $p$
and~$q$, respectively, at~$w$. If $p \geqslant q \geqslant 1$, then
$$
\label{eq-pi} c_k (G (p, q)) \leqslant c_k (G (p + 1, q - 1)), \qquad k = 0, 1, 2 \ldots, n.
$$
\end{thm}

We will apply these transformations to starlike trees. In \cite{FeYu06}, \cite{OmTa07} and \cite{GhRaTa08},
the authors proved that starlike trees are determined by their Laplacian spectrum, which means
that no two trees have equal all Laplacian coefficients.

The starlike tree $T (n_1, n_2, \ldots, n_k)$ is a
tree composed of the root $v$, and the paths $P_1, P_2,
\ldots, P_k$ of lengths $n_1, n_2, \ldots, n_k$
attached at~$v$. The number of vertices of the tree $T (n_1, n_2,
\ldots, n_k)$ equals $n = n_1 + n_2 + \ldots + n_k +
1$. The starlike tree $BS_{n, k}$ is {\em balanced} if all paths have
almost equal lengths, i.e., $|n_i - n_j| \leqslant 1$ for every $1
\leqslant i < j \leqslant k$.

Let $x = (x_1, x_2, \ldots, x_n)$ and $y = (y_1, y_2, \ldots, y_n)$
be two integer arrays of length~$n$. We say that $x$ majorize $y$ and
write $x \prec y$ if elements of these arrays satisfy following
conditions:
\begin{enumerate}
\renewcommand{\labelenumi}{(\roman{enumi})}

\item $x_1 \geqslant x_2 \geqslant \ldots \geqslant x_n$ and $y_1 \geqslant y_2 \geqslant \ldots \geqslant
y_n$,
\item $x_1 + x_2 + \ldots + x_k \geqslant y_1 + y_2 + \ldots + y_k$,
for every $1 \leqslant k < n$,
\item $x_1 + x_2 + \ldots + x_n = y_1 + y_2 + \ldots + y_n$.
\end{enumerate}

\begin{thm}
Let $p$ and $q$ be the arrays of length $k \geqslant 2$, such that $p \prec q$. Then
\begin{equation}
\label{eq:starlike}
T (p) \preceq_c T (q).
\end{equation}
\end{thm}

\begin{proof}
Let $n$ denotes the number of vertices in trees $T (p)$ and $T (q)$,
$n = p_1 + p_2 + \ldots + p_k = q_1 + q_2 + \ldots q_k$.
We will proceed by mathematical induction on the size of the array $k$.
For $k = 2$, we can directly apply transformation from Theorem \ref{thm-pi} on tree $T (q)$ several times,
in order to get $T (p)$.

Assume that the inequality (\ref{eq:starlike}) holds for all lengths less than or equal to $k$.
If there exists index $1 \leqslant m < k$ such that $p_1 + p_2 + \ldots + p_m = q_1 + q_2 + \ldots
+ q_m$, we can apply inductive hypothesis on two parts $T (q_1, q_2, \ldots, q_m)$ and $T (q_{m + 1}, q_{m + 2}, \ldots, q_k)$
and get $T (p_1, p_2, \ldots, p_m)$ and $T (p_{m + 1}, p_{m + 2}, \ldots, p_k)$.

Otherwise, we have strict inequalities $p_1 + p_2 + \ldots + p_m < q_1 + q_2 + \ldots
+ q_m$ for all indices $1 \leqslant m < k$. We can transform tree $T (q_1, q_2, \ldots, q_k)$ into
$$
T (q_1, q_2, \ldots, q_{s - 1}, q_{s} - 1, q_{s + 1}, \ldots, q_{r - 1}, q_{r} + 1, q_{r + 1}, \ldots, q_k),
$$
where $s$ is the largest index such that $q_1 = q_2 = \ldots = q_s$ and $r$ is the smallest index
such that $q_r = q_{r + 1} = \ldots = q_k$. The condition $p \prec q$ is preserved, and we can continue
until the array $q$ transforms into $p$, while at every step we decrease the Laplacian coefficients.
\end{proof}

A canonical example of majorization is

\begin{cor}
Let $T \not \cong BT_{n, k}$ be an arbitrary starlike tree with $k$ pendent paths on $n$ vertices.
Then
$$
c_k (BT_{n, k}) \leqslant c_k (T), \qquad k = 0, 1, \ldots, n.
$$
\end{cor}

%

The broom $B_{n, \Delta}$ is a tree consisting of a star $S_{\Delta
+ 1}$ and a path of length $n - \Delta - 1$ attached to an arbitrary
pendent vertex of the star (see Figure~1).
It is proven in \cite{LiGu07} that among trees with perfect matching and maximum degree equal to
$\Delta$, the broom $B_{n, \Delta}$ uniquely minimizes the largest eigenvalue of adjacency matrix.
Also it is shown that among trees with bounded degree $\Delta$, the broom has minimal
Wiener index and Laplacian-like energy \cite{St09}. In \cite{YaYe05} and \cite{YuLv06} the broom
has minimal energy among trees with fixed diameter or fixed number of pendent vertices.

\begin{figure}[ht]
  \center
  \includegraphics [width = 7.5cm]{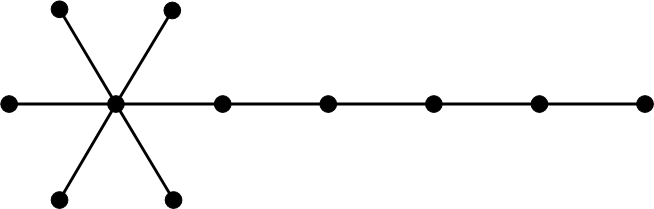}
  \caption { \textit{ The broom $B_{11, 6}$. } }
\end{figure}

For the maximum case, we have following

\begin{cor}
\label{cor-broom} Let $T \not \cong B_{n, \Delta}$ be an arbitrary
tree on $n$ vertices with the maximum vertex degree~$\Delta$. Then
$$
c_k (B_{n, \Delta}) \geqslant c_k (T), \qquad k = 0, 1, \ldots, n.
$$
\end{cor}

We can refine the above relation using Theorem \ref{thm-pi} applied on the vertex
of degree greater than $2$
$$
c_k (S_{n}) = c_k (B_{n,n-1}) \leqslant c_k (B_{n,n-2}) \leqslant \ldots \leqslant c_k (B_{n,3}) \leqslant c_k (B_{n,2}) = c_k (P_{n}),
$$
for every $k = 0, 1, \ldots n$. It follows that $B_{n, 3}$ has the second largest Laplacian
coefficients among trees on $n$ vertices.

\section{Laplacian coefficients of trees with given diameter}

Let $C(a_1, a_2, \ldots, a_{d-1})$ be a caterpillar obtained
from a path~$P_d$ with vertices $\{v_{0},v_{1},\dots,v_d\}$
by attaching $a_i$ pendent edges to vertex~$v_{i}$, $i=1,2\dots,d-1$.
Clearly, $C(a_1, a_2, \ldots, a_{d-1})$ has diameter~$d$ and
$n = d+1+\sum_{i = 1}^{d-1} a_i$.
For simplicity, denote $C_{n, d} = C (0, \ldots, 0, a_{\lfloor d / 2 \rfloor}, 0, \ldots, 0)$.
In \cite{SiMaBe08} it is shown that caterpillar $C_{n, d}$ has
minimal spectral radius (the greatest eigenvalue of adjacency
matrix) among graphs with fixed diameter.

\begin{figure}[ht]
  \center
  \includegraphics [width = 11cm]{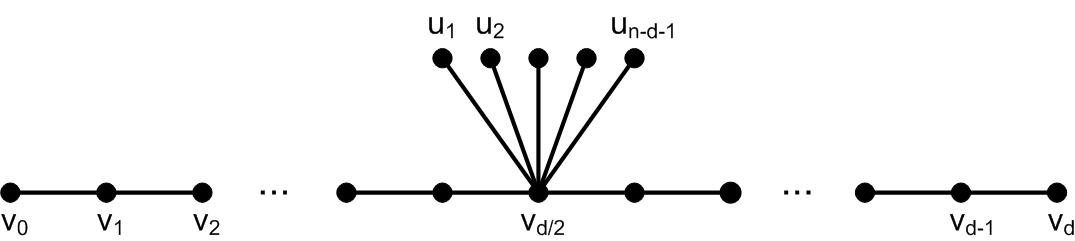}
  \caption {\textit{Caterpillar $C_{n, d}$.}}
\end{figure}

\begin{thm}
\label{thm:diameter} Among trees on $n$ vertices
and diameter $d$, caterpillar $C_{n, d}$ has minimal Laplacian
coefficient $c_k$, for every $k = 0, 1, \ldots, n$.
\end{thm}

In \cite{IlIlSt09}, the authors also considered the connected $n$-vertex graphs with fixed radius,
and proved that $C_{n, 2r - 1}$ is the extremal graph with minimal Laplacian coefficients.

Here, we give an alternative proof of Theorem \ref{thm:diameter}. Let $P = v_0 v_1 v_2 \ldots v_d$ be a path in tree $T$ of maximal
length. Every vertex $v_i$ on the path $P$ is a root of a tree $T_i$
with $a_i + 1$ vertices, that does not contain other vertices of
$P$. We apply $\sigma$-transformation (or combination of transformations
from Theorem \ref{thm-delta} and Theorem \ref{thm-pi}) on trees $T_1, T_2, \ldots,
T_{d - 1}$ to decrease coefficients $c_k$, as long as we do not get
a caterpillar $C (a_0, a_1, a_2, \ldots, a_d)$.

Let $1 \leqslant r \leqslant d - 1$ be the smallest index such that $a_r > 0$, and analogously let
$1 \leqslant s \leqslant d - 1$ be the largest index such that $a_s > 0$. We can perform $\delta$ transformation
to vertex $v_r$ or vertex $v_s$ and get a caterpillar with smaller Laplacian coefficients by moving pendent
vertices to the central vertex $v_{\lfloor d/2 \rfloor}$ of a path. After applying this algorithm,
we finally get the extremal tree~$C_{n, d}$.

If $d < n - 1$, we
can apply the transformation from Theorem~\ref{thm-pi} at the
central vertex of degree greater than $2$ and obtain $C_{n, d + 1}$.
Therefore, we have
$$
c_k (S_n) = c_k( C_{n, 2}) \leqslant c_k (C_{n, 3}) \leqslant \dots \leqslant c_k (C_{n, n - 2}) \leqslant c_k (C_{n, n - 1}) = c_k (P_n),
$$
for every $k = 0, 1, \ldots n$.
It follows that $C_{n, 3}$ has the second smallest Laplacian
coefficients among trees on $n$ vertices.

Naturally, one wants to describe $n$-vertex trees with fixed diameter
with maximal Laplacian coefficients. We have checked all trees up to $18$ vertices and
for every triple $(n, d, k)$ we found extremal trees with $n$ vertices and fixed
diameter $d$ that maximize coefficient $c_k$.
The outcome is interesting -- the extremal trees are not isomorphic (see Figure~3).

\begin{figure}[ht]
  \label{fig-dia}
  \center
  \includegraphics [width = 4cm]{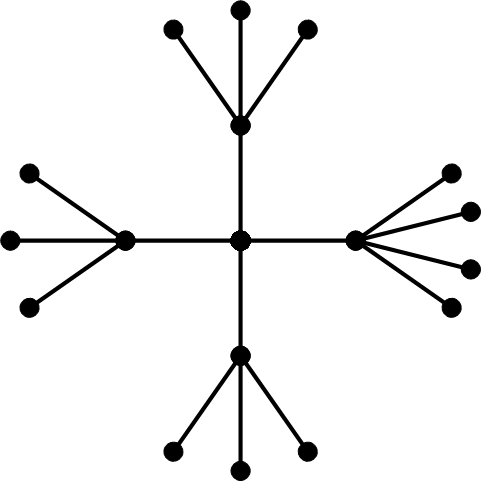}
  \hspace{2cm}
  \includegraphics [width = 4cm]{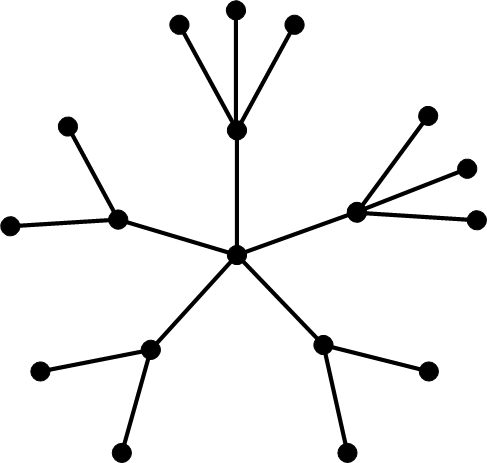}
  \caption { \textit{ Graphs with $n = 18$ and $d = 4$ that maximize $c_{16}$ and $c_{15}$. } }
\end{figure}

For $d = n - 2$, the maximum Laplacian coefficients are achieved
for $B_{n, 3}$. For $d = n - 3$, we have three potential extremal trees depicted on Figure~4 (based on
transformations from Theorem~\ref{thm-pi} and Theorem \ref{thm-delta}).

\begin{figure}[ht]
  \center
  \includegraphics [width = 11cm]{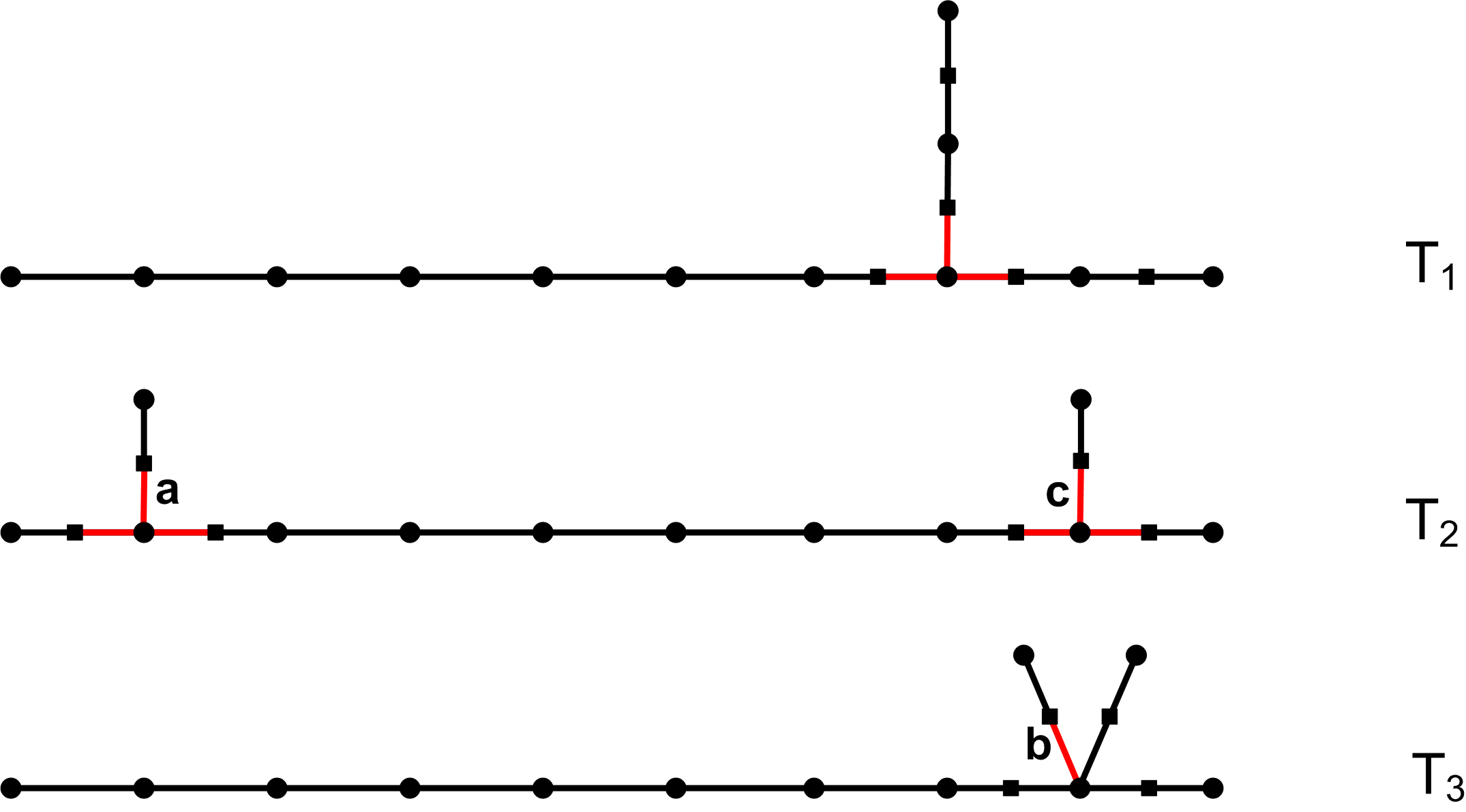}
  \caption {\textit{Extremal trees for $d = n - 3$ (squares represent subdivision vertices).}}
\end{figure}

It is easy to prove that $T_3 \preceq_c T_2$. Namely, consider two marked edges of the subdivision trees --
$a \in E (S (T_2))$ and $b \in E (S (T_3))$. Let $M$ be an arbitrary $k$-matching of $S (T_3)$. We will
construct a corresponding $k$-matching $M'$ of $S (T_2)$ and prove that $c_k (T_3) \leqslant c_k (T_2)$.
If $M$ that does not contain $b$, then $M' = M$ is also $k$-matching of the tree $S (T_2) \setminus \{a\}$.
If $M$ contains the edge $b$, then we set the edge $a$ in the corresponding matching $M'$ of $S (T_2)$.
After removing $a$ and $b$ with their neighboring edges from $S (T_2)$ and $S (T_3)$ respectively,
the decomposed graph of $S (T_3)$ is a subgraph of the decomposed graph of $S (T_2)$. Notice that
the red edge $c \in E (S (T_2))$ does not belong to any corresponding matching $M'$, 
and it follows that in this case we have an injection from the set of $k$-matchings of $S (T_3)$ that contain
the edge $b$ to the set of $k$-matchings of $S (T_2)$ that contains the edge $a$. 
Finally, we have
$c_k (T_3) < c_k (T_2)$ for $n > 7$ and $2 \leqslant k \leqslant n - 2$.

The Laplacian coefficient $c_{2} (T)$ is equal to (see \cite{FeYu06}):
$$
c_{2} (T) = 2n^2 - 5n + 3 - \frac{1}{2} \sum_{i = 1}^n d_i^2 = 2n^2 - 5n + 3 - \frac{1}{2} Z (T),
$$
where $Z (T)$ is the first Zagreb index \cite{GuDa04}. Clearly, we have
$$
Z (T_2) - Z (T_1) = (3^2 + 1^2) - (2^2 + 2^2) = 2.
$$

which means that $c_{2} (T_1) > c_{2} (T_2)$ for $n > 7$.
On the other hand, for the Wiener index $W (T)$ we have
\begin{eqnarray*}
W (T_2) - W (T_1) &=& \left ( 2 (1 + 2 + \ldots + (n - 4) + (n - 3) + 2) + n - 3 \right ) \\
&-& \left ( 2 (1 + 2 + 3 + \ldots + (n - 5) + (n - 4) + 2 + 3) + n - 1 \right) \\
&=& 2n-14,
\end{eqnarray*}
which gives $c_{n - 2} (T_2) > c_{n - 2} (T_1)$ for $n > 7$.
For $n > 7$, the pairs $(T_1, T_2)$ represent an infinite family of examples for Problem~1.

We can calculate the Laplacian coefficients of trees $T_1$ and $T_2$, by considering several cases
involving red edges on Figure 4. In \cite{IlIlSt09}, the authors proved that
for $0 \leqslant k \leqslant \lceil \frac{n}{2} \rceil$, the number
of matchings with $k$ edges for path $P_n$ is $m_k (P_n) = m (n, k) = \binom{n - k}{k}$.
After taking some independent red edges in $k$-matching, the decomposed graphs are the union of one long path and
some number of paths with lengths $2$, $3$ or $4$. Using MATHEMATICA software \cite{Wo08}, we get
\begin{eqnarray*}
c_k (T_1) &=& 6 \binom{2 n -9-k}{k-2}+8 \binom{2n-9-k}{k-1}+\binom{2n-9-k}{k}\\
&+&11\binom{2n-8-k}{k-3} + 21\binom{2n-8-k}{k-2}+6 \binom{2n-7-k}{k-4}\\
&+&20 \binom{2n-7-k}{k-3}+\binom{2n-6-k}{k-5}+5 \binom{2n-6-k}{k-4},
\end{eqnarray*}
and
\begin{eqnarray*}
c_k (T_2) &=& \binom{2n -11 - k}{k-2} + 2 \binom{2n -11 - k}{k-1} + \binom{2n -11 - k }{k}\\
&+& 4 \binom{2n -10 - k}{k-3} + 12 \binom{2n -10 - k }{k-2} + 8 \binom{2n -10 - k }{k-1} \\
&+& 6 \binom{2n -9 - k }{k-4} + 24 \binom{2n -9 - k }{k-3} + 22 \binom{2n -9 - k }{k-2}\\
&+& 4 \binom{2n -8 - k }{k-5} + 20 \binom{2n -8 - k }{k-4} + 24 \binom{2n -8 - k }{k-3}\\
&+&\binom{2n -7 - k}{k-6} + 6 \binom{2n -7 - k}{k-5} + 9 \binom{2n -7 - k}{k-4}.
\end{eqnarray*}

After some manipulations, the difference $c_k (T_2) - c_k (T_1)$ is equal to
$$
2 \cdot \frac{(2n-9-k)!}{(k-2)!(2n-2k-3)!} \cdot P (n, k),
$$
and the sign of $c_k (T_2) - c_k (T_1)$ depends only on the following expression
\begin{eqnarray*}
P (n, k) &=& -408-788 k-120 k^2-13 k^3+3 k^4+844 n+639 k n+81 k^2 n\\
&-&4 k^3 n-466 n^2-186 k n^2-6 k^2 n^2+104 n^3+16 k n^3-8 n^4.
\end{eqnarray*}

For large $n$, we can substitute $x = \frac{k}{n}$ and get the fourth degree polynomial
$$
P (x) = 3x^4 - 4 x^3 -6 x^2 +16x -8.
$$

This polynomial has only one positive real root $x_0 \approx 0.771748$, and therefore
for the Laplacian coefficient $c_k$, we have: $c_k (T_1) > c_k (T_2)$ for $k < n x_0$, and
$c_k (T_2) > c_k (T_1)$ for $k > n x_0$.

\section{Further examples for Mohar's problems}

We will use the series of $\delta$ transformations, in order to obtain a chain of
$n$-vertex trees of length $m$,
$$
S_n \cong T_0 \preceq_c T_1 \preceq_c T_2 \preceq_c \ldots \preceq_c T_{m - 1} \preceq_c T_m \cong P_n.
$$
The main idea is to move one pendent vertex attached at the center vertex of caterpillar $C_{n, d}$ to the end vertex $v_0$.
This requires $\lfloor \frac{d}{2} \rfloor$ transformations to get caterpillar $C_{n, d + 1}$, and at every step we decrease
all Laplacian coefficients.
Starting from the star $S_n$ and ending with the path $P_n$, we have
\begin{eqnarray*}
m &=& \left \lfloor \frac{2}{2} \right \rfloor + \left \lfloor \frac{3}{2} \right \rfloor + \left \lfloor \frac{4}{2} \right \rfloor +
\left \lfloor \frac{5}{2} \right \rfloor + \ldots +  \left \lfloor \frac{n - 2}{2} \right \rfloor  + \left \lfloor \frac{n - 1}{2} \right \rfloor \\
&=& 1 + 1 + 2 + 2 + \ldots + \left \lfloor \frac{n - 2}{2} \right \rfloor + \left \lfloor \frac{n - 1}{2} \right \rfloor.
\end{eqnarray*}

For $n = 2k$ it follows $m = (k - 1)^2$, while for $n = 2k + 1$ it follows $m = k (k - 1)$.
Finally, we conclude that the length of the chain is equal to
$m = \left \lfloor \frac{n - 1}{2} \right \rfloor \cdot \left \lfloor \frac{n - 2}{2} \right \rfloor$,
which is proportional to $\frac{n^2}{4}$.

\begin{table}[ht]
\centering
    \begin{tabular} {| r | r | r | r | r | r |}
    \hline
    $n$ & Tree number & Type 1 pairs & Type 2 pairs & All pairs & Percent\\
    \hline \hline
    3 & 1 & 0 & 0 & 0 & 0.00\\
    4 & 2 & 0 & 0 & 0 & 0.00 \\
    5 & 3 & 0 & 0 & 0 & 0.00 \\
    6 & 6 & 0 & 0 & 0 & 0.00 \\
    7 & 11 & 0 & 0 & 0 & 0.00 \\
    8 & 23 & 7 & 0 & 7 & 2.77 \\
    9 & 47 & 56 & 0 & 56 & 5.18 \\
    10 & 106 & 476 & 5 & 481 & 8.64 \\
    11 & 235 & 2786 & 22 & 2808 & 10.21 \\
    12 & 551 & 18857 & 230 & 19087 & 12.60 \\
    13 & 1301 & 117675 & 1756 & 119431 & 14.12 \\
    14 & 3159 & 786721 & 15203 & 801924 & 16.08 \\
    15 & 7741 & 5030105 & 109075 & 5139180 & 17.15 \\
    16 & 19320 & 33888050 & 894946 & 34782996 & 18.64 \\
    17 & 48629 & 225026865 & 6467585 & 231494450 & 19.58 \\
    18 & 123867 & 1543675765 & 50926955 & 1594602720 & 20.79\\
    \hline
    \end{tabular}
\label{tab:mohar}
\end{table}

In Table 1, we present for every $n$ between $3$ and $18$, the number of trees on $n$ vertices,
the number of pairs of trees that give affirmative answer to Problem~1 and Problem~2,
the number of all incomparable pairs and the percentage of $n$-vertex tree pairs that are incomparable.

We conclude that the smallest pair of trees that
gives the affirmative answer for Problem~1 appears on $8$ vertices, while
the smallest pair of trees for Problem~2 appears on $10$ vertices.
In \cite{ZhLvCh09} the authors did not present the smallest example for Problem~1.
Let $T_1$ and $T_2$ be two trees of order $8$ depicted on the Figure~5. It follows $T_1 \prec^1 T_2$ and $T_2 \prec^2 T_1$, since
$$
P (T_1, \lambda) = -8 x+65 x^2-190 x^3+267 x^4-196 x^5+75 x^6-14 x^7+x^8
$$
and
$$
P (T_2, \lambda) = -8 x+66 x^2-188 x^3+259 x^4-190 x^5+74 x^6-14 x^7+x^8.
$$

\begin{figure}[ht]
  \center
  \includegraphics [width = 12cm]{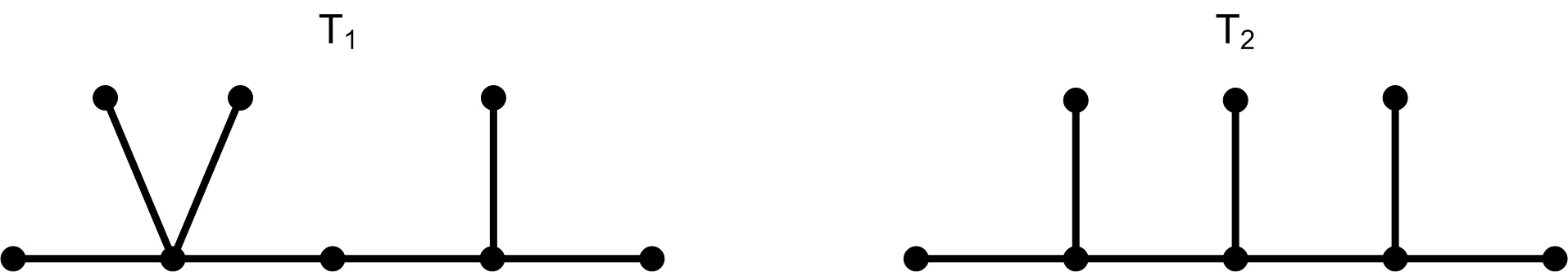}
  \caption {\textit{Trees $T_1$ and $T_2$.}}
\end{figure}

Notice that the percent of pairs of incomparable trees grows rapidly. It would be of interest
to determine the limiting ratio when $n$ tends to infinity.

\section{Laplacian coefficients of trees with perfect matchings}

It is well known that if a tree $T$ has a perfect matching,
then the perfect matching $M$ is unique. Namely, a pendent vertex $v$ has to be matched with its unique neighbor $w$,
and then $M-\{vw\}$ forms the perfect matching of tree $T-v-w$.

Let $A_{n, \Delta}$ be a $\Delta$-starlike tree $T (n - 2 \Delta, 2, 2, \ldots, 2, 1)$
consisting of a central vertex $v$,
a pendent edge, a pendent path of length $n - 2 \Delta$ and
$\Delta - 2$ pendent paths of length $2$, all attached at $v$
(see Figure~6).

\begin{figure}[ht]
  \center
  \includegraphics [width = 8cm]{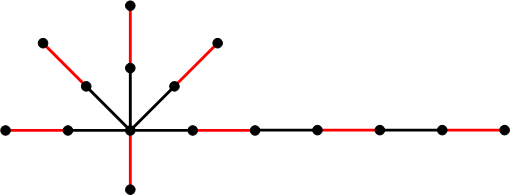}
  \caption { \textit{ The tree $A_{16, 6}$. } }
\end{figure}

\begin{thm}
The tree $A_{n, \Delta}$ has minimal Laplacian coefficients
among trees with perfect matching and maximum degree $\Delta$.
\end{thm}

\begin{proof}
Let $T$ be an arbitrary tree with perfect matching and let $v$ be a vertex of degree $\Delta$,
with neighbors $v_1, v_2, \ldots, v_{\Delta}$.
Let $T_1, T_2, \ldots, T_{\Delta}$ be the maximal subtrees
rooted at $v_1, v_2, \ldots, v_{\Delta}$, respectively,
such that neither of these trees contains $v$.
Then at most one of the numbers $|T_1|, |T_2|, \ldots, |T_{\Delta}|$ can be odd (if $T_i$ and $T_j$
have odd number of vertices, than the root vertices $v_i$ and $v_j$ will be unmatched -- which is impossible).
Actually, since the number of vertices in $T$ is even,
there exists exactly one tree among $T_1, T_2, \dots,T_{\Delta}$ with odd number of vertices.

Using Theorem~\ref{thm-pi},
we may transform each $T_i$ into a pendent path attached at $v$ --
while simultaneously decreasing all Laplacian coefficients
and keeping the existence of a perfect matching.
Assume that $T_{\Delta}$ has odd number of vertices,
while the remaining trees have even number of vertices.
We apply similar transformation to the one in Theorem~\ref{thm-pi},
but instead of moving one edge, we move two edges
in order to keep the existence of a perfect matching.
Therefore, if $p \geqslant q \geqslant 2$ then
$$
c_k (G (p, q)) \leqslant c_k (G (p + 2, q - 2)),
$$
for all $k = 0, 1, \ldots, n$.
Using this transformation,
we may reduce $T_{\Delta}$ to one vertex,
the trees $T_2, \ldots, T_{\Delta - 1}$ to two vertices,
leaving $T_1$ with $n - 2\Delta$ vertices,
and thus obtaining $A_{n,\Delta}$.
Since we have been decreasing all Laplacian coefficients simultaneously,
we conclude that $A_{n, \Delta}$ indeed has minimal Laplacian coefficients $c_k$, $k = 0, 1, \ldots, n$,
among the trees with perfect matching.
\end{proof}

If $\Delta>2$, we
can again apply Theorem~\ref{thm-pi} (by moving two vertices) at the
vertex of degree~$\Delta$ in $A_{n,\Delta}$ and obtain
$A_{n,\Delta-1}$. Thus, it follows that
$$
c_k (F_n) = c_k (A_{n,n/2}) \leqslant c_k(A_{n,n/2 - 1}) \leqslant \ldots \leqslant c_k(A_{n, 3}) \leqslant c_k(A_{n,2}) = c_k(P_{n}),
$$
holds for every $k = 0, 1, \ldots, n$.

{\bf Acknowledgement. } This work was partially done while the author was visiting the
TOPO GROUP CLUJ, leaded by Professor Mircea V. Diudea, Babe\c{s}-Bolyai University,
Faculty of Chemistry and Chemical Engineering. We gratefully acknowledge the
suggestions from the referee that helped in improving this article.

\end{document}